\pdfoutput=1
\documentclass{amsart}
\usepackage[utf8]{inputenc}

\usepackage[margin=1in]{geometry}
\usepackage[expansion=false]{microtype}

\usepackage{amsmath, amsfonts, amssymb, amsthm, amsopn}
\usepackage{eucal}

\usepackage{tikz}
\usetikzlibrary{diagrams}

\usepackage[pdfusetitle,unicode,hidelinks,bookmarksopen]{hyperref}

\newcommand{\eg}{{\sfcode`\.1000 e.g.}}

\theoremstyle{plain}
\newtheorem*{theorem*}{Theorem}

\newtheorem{theorem}{Theorem}

\newtheorem{corollary}[theorem]{Corollary}

\theoremstyle{definition}

\newtheorem{example}[theorem]{Example}
\newtheorem{question}[theorem]{Question}
\newtheorem*{acknowledgment}{Acknowledgment}

\theoremstyle{remark}
\newtheorem{remark}[theorem]{Remark}

\let\scr=\mathcal

\def\1{\mathbf 1}

\def\ph{\mathord-}

\def\L{\mathrm L}
\def\R{\mathrm R}

\let\from=\leftarrow
\let\into=\hookrightarrow

\let\tens=\otimes

\def\id{\mathrm{id}}

\def\op{\mathrm{op}}

\def\Pr{\mathcal{P}\mathrm{r}}

\def\Cat{\mathcal{C}\mathrm{at}{}}

\def\Shv{\mathrm{Shv}}

\def\Fun{\mathrm{Fun}}

\let\lim=\relax
\DeclareMathOperator*{\lim}{lim}

\DeclareMathOperator*{\colim}{colim}

\DeclareMathOperator{\Exc}{Exc}
\def\fin{\mathrm{fin}}
\def\Sp{\mathrm{Sp}}

\title{Topoi of parametrized objects}
\author{Marc Hoyois}
\date{\today}
\address{Department of Mathematics, Massachusetts Institute of Technology, Cambridge, MA, USA}
\email{hoyois@mit.edu}
\urladdr{\url{http://math.mit.edu/~hoyois/}}

\begin{document}

\begin{abstract}
	We give necessary and sufficient conditions on a presentable $\infty$-category $\scr C$ so that families of objects of $\scr C$ form an $\infty$-topos.
	In particular, we prove a conjecture of Joyal that this is the case whenever $\scr C$ is stable.
\end{abstract}

\maketitle

Let $\scr X$ be an $\infty$-topos and let $\scr C$ be a sheaf of $\infty$-categories on $\scr X$. 
We denote by
\[
\int_{\scr X}\scr C \to \scr X
\]
the cartesian fibration classified by $\scr C$. An object in $\int_\scr X\scr C$ is thus a pair $(U,c)$ with $U\in\scr X$ and $c\in\scr C(U)$.
The question we are interested in is the following:

\begin{question}\label{q:main}
	When is $\int_\scr X\scr C$ an $\infty$-topos?
\end{question}

If $\scr C$ is a presentable $\infty$-category, we will also denote by $\int_\scr X\scr C\to\scr X$ the cartesian fibration classified by the sheaf
\[
U\mapsto \Shv_\scr C(\scr X_{/U})\simeq \scr C\tens\scr X_{/U}.
\]
Joyal calls $\scr C$ an \emph{$\infty$-locus} if $\int_\scr S\scr C$ is an $\infty$-topos,\footnote{Joyal also requires an $\infty$-locus to be pointed, but it will be convenient to omit this condition.} and he conjectures that any presentable stable $\infty$-category is an $\infty$-locus \cite{Joyal}.
The motivating example, due to Biedermann and Rezk, is the $\infty$-category $\Sp$ of spectra: there is an equivalence
\[
\int_\scr X\Sp\simeq \Exc(\scr S^\fin_*,\scr X),
\]
and the right-hand side is an $\infty$-topos \cite[Remark 6.1.1.11]{HA}. More generally, for any small $\infty$-category $\scr A$ with finite colimits and a final object,
\[
\int_\scr X\Exc_*^n(\scr A,\scr S)\simeq \Exc^n(\scr A,\scr X)
\]
is an $\infty$-topos. 

This paper gives a partial answer to Question~\ref{q:main} in Theorem~\ref{thm:main}. As a corollary, we obtain a characterization of $\infty$-loci (see Corollary~\ref{cor:locus}), similar to Rezk's characterization of $\infty$-topoi,
which easily implies Joyal's conjecture (see Example~\ref{ex:stable}). The latter also follows more directly from two observations, applicable to any $\Pr^\R$-valued sheaf $\scr C$ on an $\infty$-topos $\scr X$:
\begin{itemize}
	\item If $L\colon\scr C\to\scr C$ is objectwise an accessible left exact localization functor, then the inclusion $\int_\scr XL\scr C\subset\int_\scr X\scr C$ is accessible and has a left exact left adjoint.
	 This is clear once we know that $\int_\scr X\scr C$ is presentable \cite[Theorem 1.3]{GHN}.
	\item If 
	$\scr E$ is a small $\infty$-category, there is a pullback square in $\Pr^{\L,\R}$:
	\begin{tikzmath}
		\def\colsep{2em}
		\diagram{\int_\scr X\Fun(\scr E,\scr C) & \Fun(\scr E,\int_{\scr X}{\scr C}) \\ \scr X & \Fun(\scr E,\scr X)\rlap. \\};
		\arrows (11-) edge (-12) (11) edge (21) (21-) edge (-22) (12) edge (22);
	\end{tikzmath}
\end{itemize}
These observations show that the class of $\infty$-loci is closed under accessible left exact localizations and under the formation of functor $\infty$-categories $\Fun(\scr E,\ph)$.
Since any presentable stable $\infty$-category is a left exact localization of $\Fun(\scr E,\Sp)$ for some $\scr E$ \cite[Proposition 1.4.4.9]{HA}, Joyal's conjecture holds.

The next theorem gives more intrinsic conditions on $\scr C$ implying that $\int_\scr X\scr C$ is an $\infty$-topos, and it leads to a second proof of Joyal's conjecture. Recall that a colimit in an $\infty$-category $\scr C$ with pullbacks is \emph{van Kampen} if it is preserved by the functor $\scr C\to\Cat_\infty^\op$, $c\mapsto\scr C_{/c}$, and recall that an $\infty$-category $\scr X$ is an $\infty$-topos iff it is presentable and all small colimits in $\scr X$ are van Kampen.

\begin{theorem}\label{thm:main}
	Let $\scr X$ be an $\infty$-topos and $\scr C$ a sheaf of $\infty$-categories on $\scr X$. Suppose that:
	\begin{enumerate}
		\item for every $U\in\scr X$:
		\begin{enumerate}
			\item $\scr C(U)$ is accessible and admits pushouts;
			\item pushouts in $\scr C(U)$ are van Kampen;\footnote{The existence of pullbacks in $\scr C(U)$ follows from the other assumptions, as the proof will show.}
		\end{enumerate}
		\item for every $f\colon V\to U$ in $\scr X$:
		\begin{enumerate}
			\item $f^*\colon\scr C(U)\to\scr C(V)$ preserves pushouts;
			\item $f^*$ has a left adjoint $f_!$;
			\item for every $c,0\in\scr C(V)$ with $0$ initial, the square
			\begin{tikzmath}
				\diagram{\scr C(U)_{/f_!(c)} & \scr C(U)_{/f_!(0)} \\ \scr C(V)_{/c} & \scr C(V)_{/0} \\};
				\arrows (11-) edge 
				(-12) (11) edge node[left]{$\eta^*\circ f^*$} (21) (21-) edge 
				(-22) (12) edge node[right]{$\eta^*\circ f^*$} (22);
			\end{tikzmath}
			is cartesian, where $\eta\colon \id\to f^*f_!$ is the unit of the adjunction;
		\end{enumerate}
		\item for every cartesian square
		\begin{tikzmath}
			\diagram{V' & V \\ U' & U \\};
			\arrows (11-) edge node[above]{$g$} (-12) (11) edge node[left]{$q$} (21) (21-) edge node[below]{$f$} (-22) (12) edge node[right]{$p$} (22);
		\end{tikzmath}
		in $\scr X$, the canonical transformation $q_!g^*\to f^*p_!\colon \scr C(V)\to\scr C(U')$ is an equivalence.
	\end{enumerate}
	Then $\scr C$ is a $\Pr^{\L,\R}$-valued sheaf and $\int_\scr X\scr C$ is an $\infty$-topos.
\end{theorem}

\begin{proof}
	It is clear that the coproduct of a family $(U_\alpha,c_\alpha)$ in $\int_\scr X\scr C$ is given by $(\coprod_\alpha U_\alpha,c)$, where $c=(c_\alpha)_\alpha\in\prod_\alpha\scr C(U_\alpha)\simeq \scr C(\coprod_\alpha U_\alpha)$. Given a span $(U,c)\from (W,e)\to (V,d)$, (1a) and (2b) imply that it has a pushout given by $(U\amalg_WV,u_!(c)\coprod_{w_!(e)}v_!(d))$, where $u$, $v$, and $w$ are the canonical maps to $U\amalg_WV$. Hence, $\int_\scr X\scr C$ has small colimits, and they are preserved by the projection $\int_\scr X\scr C\to\scr X$.
	By \cite[Lemma 5.4.5.5]{HTT}, this implies that $\scr C(U)$ has weakly contractible colimits, being the pullback $\{U\}\times_\scr X\int_\scr X\scr C$. By (2b), $\scr C(U)$ also has an initial object, namely $i_!(0)$ where $i\colon\emptyset\to U$ and $0$ is the unique object of $\scr C(\emptyset)\simeq *$. Hence, by (1a), $\scr C(U)$ is presentable. By (2b), we deduce that $\int_\scr X\scr C$ has pullbacks that are computed in a similar manner to pushouts.
	Note that $f^*\colon\scr C(U)\to\scr C(V)$ preserves the initial object by condition (3).
	To show that $f^*$ preserves weakly contractible colimits, it suffices to show that pullback along $(V,*)\to (U,*)$ in $\int_{\scr X}\scr C$ does, since the canonical functors $\scr C(U)\to\int_\scr X\scr C$ are conservative and preserve weakly contractible colimits.
	Once this is done, we will know that $\scr C$ is $\Pr^{\L,\R}$-valued and that $\int_\scr X\scr C$ is presentable, by \cite[Theorem 1.3]{GHN}. 
	
	It thus remains to show that colimits in $\int_\scr X\scr C$ are van Kampen. 
	The statement for coproducts is straightforward, so we only consider pushouts. We will use the factorization system on $\int_\scr X\scr C$ induced by the cocartesian fibration $\int_\scr X\scr C\to \scr X$: any map $(V,d)\to (U,c)$ 
	factors uniquely as $(V,d)\to (U,f_!(d))\to (U,c)$ where the first map is cocartesian and the second one is vertical.

	We must show that the functor
	\[\int_\scr X\scr C\to \Cat_\infty^\op,\quad (U,c)\mapsto\bigg(\int_\scr X\scr C\rlap{\bigg)}_{\;\;/(U,c)}\simeq \int_{\scr X_{/U}}\!\scr C_{/c},\] 
	preserves pushouts.
	Consider a span $(U,c)\from (W,e)\to (V,d)$ in $\int_\scr X\scr C$. Since the canonical map
	\[
	\int_{\scr X_{/U\amalg_WV}}\scr C_{/u_!c\amalg_{w_!e}v_!d}\longrightarrow\int_{\scr X_{/U}}\!\scr C_{/c}\times_{\int_{\scr X_{/W}}\!\scr C_{/e}}\int_{\scr X_{/V}}\!\scr C_{/d}
	\]
	is a map of cartesian fibrations over $\scr X_{/U\amalg_WV}\simeq \scr X_{/U}\times_{\scr X_{/W}}\scr X_{/V}$, it suffices to show that it is a fiberwise equivalence. By (2a) and (3), it suffices to consider the fiber over $U\amalg_WV$, which is
	\begin{equation}\label{eqn:vanKampen}
	\tag{\textasteriskcentered}
	\scr C(U\amalg_WV)_{/u_!c\amalg_{w_!e}v_!d} \longrightarrow \scr C(U)_{/c}\times_{\scr C(W)_{/e}}\scr C(V)_{/d}.
	\end{equation}
	Decomposing a given span in $\int_\scr X\scr C$ using the cocartesian factorization system, we see that it suffices to consider two types of spans:
	\begin{enumerate}
		\item[(i)] a vertical span $(U,c)\from (U,e)\to (U,d)$ in $\scr C(U)$;
		\item[(ii)] a span of the form $(U,f_!e)\stackrel f\from (W,e)\to (V,d)$.
	\end{enumerate}
	In case (i), the map~\eqref{eqn:vanKampen} is an equivalence by (1b). 
	In case (ii), we consider the following cube, where $P=U\amalg_WV$ and $v\colon V\to P$ is the canonical map:
	\begin{tikzmath}[
		node distance=2.5cm,
	  cross line/.style={preaction={draw=white, -,line width=6pt}}]
	  \node (A) {$\scr C(P)_{/v_!d}$};
	  \node [right of=A] (B) {$\scr C(U)_{/f_!e}$};
	  \node [below of=A] (C) {$\scr C(V)_{/d}$};
	  \node [right of=C] (D) {$\scr C(W)_{/e}$};
	  \node (A1) [right of=A, above of=A, node distance=1cm] {$\scr C(P)_{/0}$};
	  \node [right of=A1] (B1) {$\scr C(U)_{/0}$};
	  \node [below of=A1] (C1) {$\scr C(V)_{/0}$};
	  \node [right of=C1] (D1) {$\scr C(W)_{/0}\rlap.$};
 	 
	 \draw[->] (A) -- (C);
	 \draw[->] (C) -- (D);
	  \draw[->] (A) -- (A1);
	  \draw[->] (B) -- (B1);
	  \draw[->] (C) -- (C1);
	  \draw[->] (D) -- (D1);
	  \draw[->] (A1) -- (B1);
	  \draw[->] (A1) -- (C1);
	  \draw[->] (B1) -- (D1);
	  \draw[->] (C1) -- (D1);
	  \draw[->,cross line] (A) -- (B); \draw[->,cross line] (B) -- (D);
	\end{tikzmath}
	The lateral faces are cartesian by (2c). The back face is cartesian since $\scr C$ is a sheaf whose restriction maps preserve initial objects. Hence, the front face is cartesian, as desired.
\end{proof}

\begin{remark}\label{rmk:necessary}
	The conditions of Theorem~\ref{thm:main} are almost necessary. 
	Suppose that $\scr C(*)$ has initial and final objects that restrict to initial and final objects of $\scr C(U)$ for every $U\in\scr X$. 
	Under this mild assumption, if $\int_\scr X\scr C$ is an $\infty$-topos, then conditions (1) and (2a) hold with ``pushouts'' replaced by ``weakly contractible colimits'', and condition (2c) and (3) hold provided (2b) does. Thus, (2b) is the only condition that may not be necessary for $\int_\scr X\scr C$ to be an $\infty$-topos. However, if $\scr C$ is the sheaf associated with a presentable $\infty$-category, it is clear that (1a), (2a), (2b), and (3) always hold; in that case, therefore, $\int_\scr X\scr C$ is an $\infty$-topos iff (1b) and (2c) hold.
\end{remark}

\begin{example}\label{ex:sheaf}
	Let $\scr C$ be a sheaf on $\scr X$ satisfying the assumptions of Theorem~\ref{thm:main}. Then the following sheaves of $\infty$-categories on $\scr X$ also satisfy the assumptions of Theorem~\ref{thm:main}:
	\begin{enumerate}
		\item the subsheaf $L\scr C\subset\scr C$ for any $L\colon\scr C\to\scr C$ that is objectwise an accessible left exact localization functor;
		\item the sheaf $\Fun(\scr E,\scr C)$ for any small $\infty$-category $\scr E$;
		\item 
		the sheaf $\Exc^n(\scr A,\scr C)$ for any $n\geq 0$ and any small $\infty$-category $\scr A$ with finite colimits and a final object (this follows from the previous two cases, noting that $\scr C$ is objectwise differentiable);
		\item the sheaves $\scr C_{/c}$ and $\scr C_{c/}$ for any $c\in\scr C(*)$.
	\end{enumerate}
\end{example}

Specializing to sheaves of the form $U\mapsto\Shv_{\scr C}(\scr X_{/U})$, we obtain the following characterization of $\infty$-loci:

\begin{corollary}\label{cor:locus}
	Let $\scr C$ be a presentable $\infty$-category. The following assertions are equivalent:
	\begin{enumerate}
		\item $\scr C$ is an $\infty$-locus. 
		\item For every $\infty$-topos $\scr X$, $\int_\scr X\scr C$ is an $\infty$-topos.
		\item Weakly contractible colimits in $\scr C$ are van Kampen.
		\item Pushouts in $\scr C$ are van Kampen, and for every $\infty$-groupoid $A$ and every functor $F\colon A^{\triangleleft}\to\scr C$ sending the initial vertex to the initial object of $\scr C$, the colimit of $F$ is van Kampen.
	\end{enumerate}
\end{corollary}

\begin{proof}
	The equivalence of $(1)$ and $(2)$ follows from
	\[\int_\scr X\scr C\simeq \scr X\otimes\int_\scr S\scr C.\]
	 Note that the functor $\scr C\to \int_{\scr S}\scr C$, $c\mapsto (*,c)$, is fully faithful and preserves limits and weakly contractible colimits. The implication $(1)\Rightarrow (3)$ follows immediately, and $(3)\Rightarrow(4)$ is obvious. To prove $(4)\Rightarrow(1)$, we apply Theorem~\ref{thm:main}. The only nontrivial condition to check is (2c): given a morphism $f\colon V\to U$ in $\scr S$ and a functor $F\colon V\to\scr C$, we must show that the square
	\begin{tikzmath}
		\def\colsep{1.5em}
		\diagram{\Fun(U,\scr C)_{/f_!F} & \Fun(U,\scr C_{/0}) \\ \Fun(V,\scr C)_{/F} & \Fun(V,\scr C_{/0}) \\};
		\arrows (11-) edge (-12) (11) edge (21) (21-) edge (-22) (12) edge (22);
	\end{tikzmath}
	is cartesian. But this square is the limit over $u\in U$ of the squares
	\begin{tikzmath}
		\def\colsep{.8em}
		\diagram{\scr C_{/\colim_{v\in f^{-1}(u)} F(v)} & \scr C_{/0} \\ \lim_{v\in f^{-1}(u)}\scr C_{/F(v)} & \lim_{v\in f^{-1}(u)}\scr C_{/0}\rlap, \\};
		\arrows (11-) edge (-12) (11) edge (21) (21-) edge (-22) (12) edge (22);
	\end{tikzmath}
	which are cartesian by assumption.
\end{proof}

\begin{example}
	An $\infty$-category is an $\infty$-topos iff it is an $\infty$-locus with a strictly initial object.
\end{example}

\begin{example}\label{ex:stable}
	If $\scr C$ is a stable $\infty$-category, any weakly contractible colimit in $\scr C$ is van Kampen. This follows easily from the fact that cartesian squares in $\scr C$ are preserved by colimits.
	 In particular, we obtain another proof that any presentable stable $\infty$-category is an $\infty$-locus.
\end{example}

\begin{example}
	Let $\scr C'\to\scr C$ be a conservative functor between presentable $\infty$-categories that preserves pullbacks and weakly contractible colimits.
	If $\scr C$ is an $\infty$-locus, so is $\scr C'$.
	For example:
	\begin{enumerate}
		\item If $\scr Y$ is an $\infty$-topos, the $\infty$-category $\scr Y_*^{\geq \infty}$ of pointed $\infty$-connective objects of $\scr Y$ is an $\infty$-locus.
		\item If $\scr C$ is an $\infty$-locus and $\scr A$ is a small $\infty$-category with finite colimits and a final object, the $\infty$-category $\Exc^n_*(\scr A,\scr C)$ of reduced $n$-excisive functors from $\scr A$ to $\scr C$ is an $\infty$-locus (use Example~\ref{ex:sheaf} (3)).
		\item If $\scr C$ is an $\infty$-locus and $T\colon\scr C\to\scr C$ is an accessible pullback-preserving comonad, the $\infty$-category $\mathrm{coAlg}_T(\scr C)$ of $T$-coalgebras is an $\infty$-locus. Similarly, if $T\colon\scr C\to\scr C$ is a monad that preserves weakly contractible colimits, then the $\infty$-category $\mathrm{Alg}_T(\scr C)$ of $T$-algebras is an $\infty$-locus.
	\end{enumerate}
\end{example}

We conclude this note by considering an amusing family of $\infty$-topoi.
Let $\scr C$ be an $\infty$-locus in which every truncated object is contractible, \eg, a presentable stable $\infty$-category. The projection $\int_\scr X\scr C\to\scr X$ admits left and right adjoints given by $U\mapsto(U,\emptyset)$ and $U\mapsto (U,*)$, so that the functor
\[
\scr X\into\int_\scr X\scr C,\quad U\mapsto (U,*),
\]
is an essential geometric embedding. Moreover, it is clear that an object $(U,c)\in\int_\scr X\scr C$ is hypercomplete if and only if $U$ is hypercomplete and $c\simeq *$; in other words, $\scr X$ is a cotopological localization of $\int_\scr X\scr C$ \cite[Definition 6.5.2.17]{HTT}.
In particular, $\int_\scr S\scr C$ is an $\infty$-topos whose hypercompletion is $\scr S$ and whose full subcategory of $\infty$-connective objects is $\scr C$. 

\begin{acknowledgment}
	I thank anonymous MathOverflow user ``user84563'' for asking a question about loci that led me to write this note.
\end{acknowledgment}

\providecommand{\bysame}{\leavevmode\hbox to3em{\hrulefill}\thinspace}

\end{document}